\documentclass[oneside,american]{amsart}
\usepackage[T1]{fontenc}
\usepackage[latin9]{inputenc}
\usepackage{geometry}
\usepackage{color}
\usepackage{babel}
\usepackage{mathrsfs}
\usepackage{mathtools}
\usepackage{amstext}
\usepackage{amsthm}
\usepackage{amssymb}
\usepackage{setspace}
\usepackage[unicode=true,pdfusetitle,
 bookmarks=true,bookmarksnumbered=false,bookmarksopen=false,
 breaklinks=false,pdfborder={0 0 1},backref=false,colorlinks=false]
 {hyperref}

\makeatletter

\providecommand{\tabularnewline}{\\}

\numberwithin{equation}{section}
\numberwithin{figure}{section}
\theoremstyle{plain}
\newtheorem{thm}{\protect\theoremname}[section]
\theoremstyle{definition}
\newtheorem{defn}[thm]{\protect\definitionname}
\theoremstyle{plain}
\newtheorem{lem}[thm]{\protect\lemmaname}
\ifx\proof\undefined
\newenvironment{proof}[1][\protect\proofname]{\par
\normalfont\topsep6\p@\@plus6\p@\relax
\trivlist
\itemindent\parindent
\item[\hskip\labelsep\scshape #1]\ignorespaces
}{%
\endtrivlist\@endpefalse
}
\providecommand{\proofname}{Proof}
\fi
\theoremstyle{remark}
\newtheorem{rem}[thm]{\protect\remarkname}
\theoremstyle{plain}
\newtheorem{cor}[thm]{\protect\corollaryname}
\theoremstyle{definition}
\newtheorem{example}[thm]{\protect\examplename}

\usepackage{pgf,tikz} \usetikzlibrary{arrows}

\usepackage{mathrsfs}
\usepackage[bottom]{footmisc}
\makeatletter
\let\@@pmod\pmod
\DeclareRobustCommand{\pmod}{\@ifstar\@pmods\@@pmod}
\def\@pmods#1{\mkern4mu({\operator@font mod}\mkern 6mu#1)}
\makeatother

\usepackage{fullpage}
\date{}

\makeatother

\providecommand{\corollaryname}{Corollary}
\providecommand{\definitionname}{Definition}
\providecommand{\examplename}{Example}
\providecommand{\lemmaname}{Lemma}
\providecommand{\remarkname}{Remark}
\providecommand{\theoremname}{Theorem}

\begin{document}

\title{Efficient Point-Counting Algorithms for Superelliptic Curves}

\author{Matthew Hase-Liu and Nicholas Triantafillou}
\begin{abstract}
{\normalsize{}In this paper, we present efficient algorithms for computing
the number of points and the order of the Jacobian group of a superelliptic
curve over finite fields of prime order $p.$ Our method employs the
Hasse-Weil bounds in conjunction with the Hasse-Witt matrix for superelliptic
curves, whose entries we express in terms of multinomial coefficients.
We present a fast algorithm for counting points on specific trinomial
superelliptic curves and a slower, more general method for all superelliptic
curves. For the first case, we reduce the problem of simplifying the
entries of the Hasse-Witt matrix modulo $p$ to a problem of solving
quadratic Diophantine equations. For the second case, we extend Bostan
et al.'s method for hyperelliptic curves to general superelliptic
curves. We believe the methods we describe are asymptotically the
most efficient known point-counting algorithms for}\textcolor{red}{\normalsize{}
}{\normalsize{}certain families of trinomial superelliptic curves.}
\pagebreak{}
\end{abstract}

\maketitle

\section{Introduction}

In this paper, we present and prove asymptotics for the fastest known
algorithms for counting the number of points on certain families of
possibly singular plane curves. A central problem in number theory
is the study of rational solutions of polynomial equations. Even before
the development of algebra, the Greeks were interested in systematically
determining all rational solutions to the Pythagorean equation: $a^{2}+b^{2}=c^{2}.$
More recently, Andrew Wiles proved Fermat's Last Theorem, which states
that Fermat's equation, $a^{n}+b^{n}=c^{n},$ has no nontrivial rational
solutions \textemdash{} a problem that had withstood over 350 years
of effort by mathematicians. 

To study points on a curve defined by a polynomial equation, it is
often helpful to keep track of the field where they ``live.'' Solutions
to Fermat's equation are called $\mathbb{Q}$-rational because we
are concerned with solutions over the field $\mathbb{Q}$. More generally,
solutions $(x,y)$ to polynomial equations $f(x,y)=0$ are called
$K$-rational points, if both $x$ and $y$ are in a field $K$. In
particular, mathematicians are concerned with $\mathbb{F}_{p^{r}}$-rational
points (where $p$ is some prime) because they 1) are easier to study
than $\mathbb{Q}$-rational points, 2) have analogues to complex analysis,
and 3) are useful in investigating and understanding $\mathbb{Q}$-rational
points \cite{Koblitz}. In this paper, we focus on counting $\mathbb{F}_{p}$-points
on the curve defined by a particular type of irreducible polynomial
equation:
\[
0=y^{a}-x^{b}f(x)=y^{a}-\left(m_{c}x^{b+c}+\ldots+m_{0}x^{b}\right).
\]
 The curve $C$ defined by $y^{a}-x^{b}f(x)=0$ is commonly known
as a superelliptic curve.

Advances in the field of modern algebraic geometry have brought forth
a wealth of techniques to further the study of $\mathbb{F}_{p^{r}}$-rational
points on plane curves.\footnote{A plane curve is the loci of a single bivariate polynomial equation
over a field. } We define $\#C\left(\mathbb{F}_{p^{r}}\right)$ to be the number
of $\mathbb{F}_{p^{r}}$-rational points on a smooth projective curve
$C$ and $\#J_{C}\left(\mathbb{F}_{p}\right)$ to be the number of
points in the Jacobian of $C$, an abelian group canonically associated
with the curve. In the early 1940s, André Weil proved the Riemann
hypothesis for curves over finite fields, which reveals a great deal
about the relationship between $\#C\left(\mathbb{F}_{p^{r}}\right)$
and a curve's zeta function. A direct consequence is the Hasse-Weil
bounds, which state that 
\begin{equation}
\left|p^{r}+1-\#C\left(\mathbb{F}_{p^{r}}\right)\right|\le2g\sqrt{p^{r}}\textrm{ and }\left(\sqrt{p^{r}}-1\right)^{2g}\le\#J_{C}\left(\mathbb{F}_{p^{r}}\right)\le\left(\sqrt{p^{r}}+1\right)^{2g},\label{eq:hasse bound}
\end{equation}
where $g$ is the genus of $C$ \cite{japanese}.

In this paper, we describe fast algorithms for computing $\#C\left(\mathbb{F}_{p}\right)$
and $\#J_{C}\left(\mathbb{F}_{p}\right)$ for trinomial (when $y^{a}-x^{b}f(x)$
consists of three monomials) and general superelliptic curves, even
if they are singular at the origin. To do so, we employ a matrix associated
with $C$ that encodes information about $\#C\left(\mathbb{F}_{p}\right)$
and $\#J_{C}\left(\mathbb{F}_{p}\right)$ modulo $p$ (called the
Hasse-Witt matrix) and explicitly compute its entries in terms of
multinomial coefficients. We show that these multinomial coefficients
are actually binomial for trinomial superelliptic curves, which allows
us to reduce the problem of simplifying the Hasse-Witt matrix modulo
$p$ to a problem of solving Diophantine equations of the form $x^{2}+dy^{2}=m$.
Assuming $p>16g^{2}$ and $\gcd(ac,p-1)=3,4,6\textrm{ or }8,$ we
use both this simplified Hasse-Witt matrix and the Hasse-Weil bounds
to compute $\#C\left(\mathbb{F}_{p}\right)$ probabilistically in
$O(\textrm{M}(\log p)\log p)$ time, deterministically in $O\left(\textrm{M}(\log p)\log^{2}p\log\log p\right)$
(assuming the generalized Riemann hypothesis), and deterministically
in $O\left(\textrm{M}(\log^{3}p)\frac{\log^{2}p}{\log\log p}\right)$
time, where $\textrm{M}(n)$ is the time needed to multiply two $n$-digit
numbers. Similarly, we can compute $\#J_{C}\left(\mathbb{F}_{p}\right)$
with the same running times for $g=2$ and with running time $O\left(\sqrt{p}\right)$
if $g=3$, assuming that $a$ also divides $p-1$. For general superelliptic
curves, we extend the work of Bostan et al. \cite{Boston} to simplify
the Hasse-Witt matrix modulo $p$ in $O\left(\textrm{M}\left(\sqrt{p}\log p\right)\right)$
time. Again, assuming $p>16g^{2}$ (and additionally $g=2\textrm{ or 3}$
for $\#J_{C}\left(\mathbb{F}_{p}\right)$), we use this simplified
Hasse-Witt matrix and the Hasse-Weil bounds to also compute $\#C\left(\mathbb{F}_{p}\right)$
and $\#J_{C}\left(\mathbb{F}_{p}\right)$ in $O\left(\textrm{M}\left(\sqrt{p}\log p\right)\right)$
time.

Our work has salient applications in mathematics, cryptography, and
coding theory. The Hasse-Witt matrix can be used to find the L-polynomial
of a curve, which is intimately connected to the zeta function of
a curve \textemdash{} the central object of study in the Weil-conjectures.
In addition, Fité et al.'s classification of Sato-Tate distributions
in curves of genus two required extensive back-and-forth between theoretical
work and computations of $\#C\left(\mathbb{F}_{p}\right)$ \cite{Sato-tate}.
Fast point-counting algorithms will be essential in studying the Sato-Tate
groups of curves of higher genus. Moreover, $J_{C}\left(\mathbb{F}_{p}\right)$
is a finite abelian group, so a public-key cryptosystem can be constructed
with it. Using the Jacobian of a curve of higher genus (typically
3 or 4) allows us to achieve a comparable level of security using
a smaller field size, even after taking into account known attacks
on the discrete logarithm problem for curves of genus $g>2$. For
example, hyperelliptic curve cryptosystems can have field sizes half
of those used in elliptic curve cryptosystems \cite{japanese}. To
build such cryptosystems, computing $\#J_{C}\left(\mathbb{F}_{p}\right)$
is essential to ascertain which Jacobians are cryptographically secure
in practice.\footnote{For a Jacobian to be secure, $\#J_{C}\left(\mathbb{F}_{p}\right)$
is required to be a product of a very small integer and a very large
prime.} 

The organization of this paper is as follows. In Section 2, we provide
background information about regular differentials on a curve, the
Cartier operator, and the Hasse-Witt matrix. In Section 3, we use
the Cartier operator to compute the entries of the Hasse-Witt matrix
in terms of multinomial coefficients. In Section 4, we give fast point-counting
algorithms and their time complexities for both trinomial and general
superelliptic curves. We conclude in Section 5 with a brief discussion
on future work.

\section{Background}

In this section, we recall several facts about regular differentials
on a curve, the Cartier operator, the Hasse-Witt matrix, and $\#C\left(\mathbb{F}_{p}\right)$
and $\#J_{C}\left(\mathbb{F}_{p}\right)$.

\subsection{Basic Facts }

In this paper, we consider smooth projective plane curves $C$ with
affine model $y^{a}=x^{b}f(x)=x^{b}\left(m_{c}x^{c}+m_{c-1}x^{c-1}+\ldots+m_{0}\right)$
over a finite field $\mathbb{F}_{p}$, where $p\ne2,3$ is a prime,
$y^{a}-x^{b}f(x)$ is irreducible, $f(x)\in\mathbb{F}_{p}[x],$ $a,c\ge2,b\ge0,m_{c}\ne0,m_{0}\ne0,$
and $p>a,b+c$. 

Let $P$ be the convex hull of the set of lattice points in the interior
of the Newton polygon of $g(x,y)=y^{a}-x^{b}f(x).$ If we take a side
of $P$ and label the lattice points that lie on it $0,1,\ldots,s,$
construct a ``side polynomial'' $h(t)$ with degree $s+1$ such
that the coefficient of $t^{i}$ is the coefficient of $g$ corresponding
to the lattice point $i$. 

We impose the following two conditions on our superelliptic curves:
\begin{enumerate}
\item The curves $C/\mathbb{Z}$ and $C/\mathbb{F}_{p}$ have the same genus.
\item The ``side polynomials'' are square-free.
\end{enumerate}
We denote the set of points on a curve $C$ over the field $\mathbb{F}_{p}$
as $C\left(\mathbb{F}_{p}\right)$ and the Jacobian variety as $J_{C}\left(\mathbb{F}_{p}\right)$. 

\subsection{Regular Differentials and the Cartier Operator}

Given a smooth projective variety $C$ with coordinate ring $k\left[C\right]$,
we can define $\Omega\left(C\right)$, the $k\left[C\right]$-module
of regular differentials as elements of the form $\sum_{i}g_{i}df_{i}$
with $f_{i},g_{i}\in k\left[C\right]$ subject to
\begin{itemize}
\item $d\left(f+g\right)=df+dg$;
\item $d\left(fg\right)=fdg+gdf$;
\item $da=0$ if $a\in k$.
\end{itemize}
It is worth noting that $\Omega\left(C\right)$ is a $g$-dimensional
vector space over the field $k,$ where $g$ is the genus of $C$.

Let $\omega$ be in $\Omega\left(C\right)$, where $p$ is the prime
corresponding to the finite field $\mathbb{F}_{p}$ over which $C$
is defined. One can always find $a_{0},a_{1},\ldots,a_{p-1}$ such
that $\omega=\left(a_{0}^{p}+a_{1}^{p}x+\ldots+a_{p-1}^{p}x^{p-1}\right)dx$
\cite{cartier}. We can then define the Cartier operator as follows.
\begin{defn}
The Cartier operator $\mathscr{C}$ is a map from $\Omega\left(C\right)$
to itself such that 
\[
\mathcal{\mathscr{C}\left(\omega\right)}=a_{p-1}dx.
\]

The following are useful properties of the Cartier operator that we
use later in writing down the Hasse-Witt matrix \cite{Cartier properties}.
\end{defn}
\begin{lem}
\label{lem:cartier lem}If $F$ is the function field of $C,$ then
for all $\omega\in\Omega\left(C\right)$ and $f\in F/\mathbb{F}_{p},$
\end{lem}
\begin{itemize}
\item $\mathscr{C}(f^{p}\omega)=f\mathscr{C}(\omega)$;
\item $\mathscr{C}^{n}\left(x^{j}dx\right)=\begin{cases}
0, & \textrm{if }p^{n}\nmid j+1;\\
x^{(j+1)/p^{n}-1}, & \textrm{otherwise}.
\end{cases}$
\end{itemize}

\subsection{Finding $\#C\left(\mathbb{F}_{p}\right)$ and $\#J_{C}\left(\mathbb{F}_{p}\right)$
with the Hasse-Witt Matrix}

The Hasse-Witt matrix is the matrix associated with a smooth curve
over some finite field and the action of the Frobenius endomorphism
with respect to a basis of regular differentials for $\Omega\left(C\right)$. 
\begin{defn}
Let $\omega_{1},\ldots,\omega_{g}$ be a basis for $\Omega\left(C\right)$.
Then, define the Hasse-Witt matrix $A_{p}\left(C\right)$ as $\left[a_{i,j}^{1/p}\right]$,
where 
\[
\mathscr{C}\left(\omega_{i}\right)=\sum_{j=1}^{g}a_{i,j}\omega_{j}.
\]
\end{defn}
Since we are dealing with only finite fields with prime order, matrix
$\left[a_{i,j}^{1/p}\right]$ is the same as matrix $\left[a_{i,j}\right]$
by Fermat's little theorem. The following theorem shows how this matrix
encodes information about $\#C\left(\mathbb{F}_{p}\right)$ and $\#J_{C}\left(\mathbb{F}_{p}\right)$
\cite{Manin}.
\begin{thm}
Let $A_{p}\left(C\right)$ be the Hasse-Witt matrix of a smooth curve
$C$ and $\chi_{p}\left(t\right)$ be the characteristic polynomial
of the Frobenius endomorphism. Then 
\[
\chi_{p}\left(t\right)\equiv\left(-1\right)^{g}t^{g}\det\left(A_{p}(C)-tI\right)\pmod*{p},
\]
where $I$ is the identity matrix.
\end{thm}
It is well known that $\chi_{p}\left(t\right)$ is a polynomial of
degree twice that of the genus $g$ and that the coefficient of $x^{2g-1}$
in $\chi_{p}\left(t\right)$ is simply the negative of the trace of
$A_{p}\left(C\right)$ by the Cayley\textendash Hamilton theorem.
The trace of Frobenius, $\textrm{tr}\left(A_{p}\left(C\right)\right),$
satisfies
\[
\textrm{tr}\left(A_{p}\left(C\right)\right)=p+1-\#C\left(\mathbb{F}_{p}\right).
\]
We can thus compute $\#C\left(\mathbb{F}_{p}\right)$ and $\#J_{C}\left(\mathbb{F}_{p}\right)$
modulo $p$ by noting that 
\begin{equation}
\#C\left(\mathbb{F}_{p}\right)\equiv1-\textrm{tr}\left(A_{p}\left(C\right)\right)\pmod*{p}\label{eq:C}
\end{equation}
and 
\begin{equation}
\#J_{C}\left(\mathbb{F}_{p}\right)=\chi_{p}\left(1\right)\equiv\left(-1\right)^{g}\det\left(A_{p}(C)-I\right)\pmod*{p}.\label{eq:J}
\end{equation}
If $p$ is sufficiently large, we can then uniquely determine $\#C\left(\mathbb{F}_{p}\right)$
and narrow down the candidates for $\#J_{C}\left(\mathbb{F}_{p}\right)$
using the Hasse-Weil bounds.

\section{Computing Hasse-Witt Matrices\label{sec:Computing-Hasse-Witt-Matrices}}

In this section, we express the entries of the Hasse-Witt matrix of
the superelliptic curve 
\begin{equation}
C:y^{a}=x^{b}f(x)=x^{b}\left(m_{c}x^{c}+m_{c-1}x^{c-1}+\ldots+m_{0}\right)\label{eq:superelliptic}
\end{equation}
over $\mathbb{F}_{p}$ explicitly as a sum of multinomial coefficients,
where $p$ is some prime bigger than $a$ and $b+c$, $f(0)\ne0$,
and $\deg(f)=c$. It is well-known \cite{RS book} that 
\[
\left\{ \omega_{i,j}=\dfrac{x^{i-1}y^{j-1}}{y^{a-1}}\middle|(i,j)\in\mathcal{N}\right\} 
\]
is a basis of regular differentials for $\Omega\left(C\right),$ where
$\mathcal{N}$ is the set of lattice points in the interior of the
Newton polygon of $y^{a}-x^{b}f(x)$:

\[
\mathcal{N}=\left\{ (i,j)\middle|\begin{array}{c}
1\le i\le b\textrm{ and }\left\lfloor -\dfrac{a}{b}i+a\right\rfloor +1\le j\le\left\lceil -\dfrac{a}{b+c}i+a\right\rceil -1\\
\textrm{or}\\
b+1\le i\le b+c-1\textrm{ and }1\le j\le\left\lceil -\dfrac{a}{b+c}i+a\right\rceil -1
\end{array}\right\} .
\]

The following lemmas will be used in our proof of Theorem \ref{thm:The-Hasse-Witt-matrix}.
\begin{lem}
\label{lem:1}If $\gcd(p,a)=1$ and $1\le j\le a$, there exist unique
integers $u_{j}$ and $v_{j}$ such that $j-1+(a-1)(p-1)=p(u_{j}-1)+av_{j}$,
where $1\le u_{j}\le a$ and $0\le v_{j}\le p-1$.
\end{lem}
\begin{proof}
Consider the Diophantine equation $p(x-1)+ay=j-1+(a-1)(p-1)$. If
we take this equation modulo $p$, we get $y\equiv a^{-1}\left(j-1+\left(p-1\right)\left(a-1\right)\right)\pmod*{p}.$
Let $v_{j}$ be the unique integer such that $v_{j}\equiv y\pmod*{p}$
and $0\le v_{j}\le p-1$. Then,  it remains to show that $u_{j}=\frac{j-1+(a-1)(p-1)-av_{j}}{p}+1$
is between $1$ and $a$, inclusive. Note that $av_{j}\equiv j-1+(a-1)(p-1)\pmod*{p}$
and $0\le av_{j}\le ap-a$, so $av_{j}\le j-1+(a-1)(p-1)$. Thus,
$u_{j}\ge1$. Also, since $j\le a$, we have $j-1+(a-1)(p-1)-av_{j}\le pa-p$,
so $u_{j}\le a$. 
\end{proof}
\begin{rem}
It turns out that the $u_{j}$ are pairwise distinct, i.e. $u_{j}=u_{j'}\Rightarrow j=j'$.
To see this, note that $u_{j}=u_{j'}$ implies that $a\left(v_{j}-v_{j'}\right)=j-j',$
so $a|j-j'$. However, $1\le j,j'\le a-1$, so $0\le j-j'\le a-2$.
Thus, $j=j'$. 
\end{rem}
\begin{defn}
\begin{doublespace}
Let $(i,j)\in\mathcal{N}$ and $k_{0},k_{1},\ldots$, and $k_{c}$
be integers such that $k_{0}+k_{1}+\ldots+k_{c}=v_{j},0\le k_{0},\ldots,k_{c}\le v_{j},$
and $p|bv_{j}+ck_{c}+(c-1)k_{c-1}+\ldots+k_{1}+i$. Then, define
\end{doublespace}

\[
s_{i,j}(k_{0},\ldots,k_{c}):=\dfrac{bv_{j}+ck_{c}+(c-1)k_{c-1}+\ldots+k_{1}+i}{p}.
\]
\end{defn}
\begin{lem}
\label{lem:2}$u_{j}\le\left\lceil -\frac{a}{b+c}s_{i,j}(k_{0},\ldots,k_{c})+a\right\rceil -1$.
\end{lem}
\begin{proof}
Note that $\left\lfloor \frac{a}{b+c}s_{i,j}(k_{0},\ldots,k_{c})\right\rfloor +\frac{j}{p}=\left\lfloor \frac{a}{b+c}\frac{bv_{j}+ck_{c}+(c-1)k_{c-1}+\ldots+k_{1}+i}{p}\right\rfloor +\frac{j}{p}\le\left\lfloor \frac{a\left((b+c)v_{j}+i\right)}{p(b+c)}\right\rfloor +\frac{a-\left\lfloor \frac{ai}{b+c}\right\rfloor -1}{p}$\\
$\le\frac{av_{j}+a+\left\{ \frac{ai}{b+c}\right\} -1}{p},$ so $\left\lfloor \frac{a}{b+c}s_{i,j}(k_{0},\ldots,k_{c})\right\rfloor \le\frac{av_{j}+a-j+\left\{ \frac{ai}{b+c}\right\} -1}{p}<\frac{av_{j}+a-j}{p}.$
However, $\left\lfloor \frac{a}{b+c}s_{i,j}(k_{0},\ldots,k_{c})\right\rfloor $
and $\frac{av_{j}+a-j}{p}=a-u_{j}$ are both integers, so $\left\lfloor \frac{a}{b+c}s_{i,j}(k_{0},\ldots,k_{c})\right\rfloor \le\frac{av_{j}+a-j}{p}-1.$
Using the fact that $-\left\lfloor x\right\rfloor =\left\lceil -x\right\rceil $,
we find that $u_{j}=\frac{j-av_{j}-a}{p}+a\le\left\lceil -\frac{a}{b+c}s_{i,j}(k_{0},\ldots,k_{c})+a\right\rceil -1,$
as desired. 
\end{proof}
\begin{lem}
\label{lem:3}$\left\lfloor -\frac{a}{b}s_{i,j}(k_{0},\ldots,k_{c})+a\right\rfloor +1\le u_{j}$
for $1\le s_{i,j}(k_{0},\ldots,k_{c})\le b$.
\end{lem}
\begin{proof}
We consider two cases: $1\le i\le b$ and $b+1\le i\le b+c-1$. If
$1\le i\le b$, note that ${\textstyle \frac{av_{j}+a-j}{p}\le\frac{av_{j}+\left\lceil \frac{ai}{b}\right\rceil -1}{p}}$\\
$<\frac{av_{j}+\frac{ai}{b}}{p}\le\frac{a(bv_{j}+ck_{c}+(c-1)k_{c-1}+\ldots+k_{1}+i)}{pb}=\left\lceil \frac{a}{b}s_{i,j}(k_{0},\ldots,k_{c})\right\rceil .$
Analogously, if $b+1\le i\le b+c-1$, we have $\frac{av_{j}+a-j}{p}<\left\lceil \frac{a}{b}s_{i,j}(k_{0},\ldots,k_{c})\right\rceil .$
In both situations, since $\frac{av_{j}+a-j}{p}$ and $\left\lceil \frac{a}{b}s_{i,j}(k_{0},\ldots,k_{c})\right\rceil $
are integers, we have $\left\lceil \frac{a}{b}s_{i,j}(k_{0},\ldots,k_{c})\right\rceil \ge\frac{av_{j}+a-j}{p}+1.$
Thus $u_{j}=\frac{j-av_{j}-a}{p}+a\ge\left\lfloor -\frac{a}{b}s_{i,j}(k_{0},\ldots,k_{c})+a\right\rfloor +1,$
as desired.
\end{proof}
\begin{lem}
\label{lem:4}$1\le u_{j}$ for $c+1\le s_{i,j}(k_{0},\ldots,k_{c})\le b+c-1$.
\end{lem}
\begin{proof}
Note that $u_{j}=\frac{j-av_{j}-a}{p}+a\ge\frac{j-a(p-1)-a}{p}+a=\frac{j}{p}>0.$
We know $u_{j}$ is an integer, so $u_{j}\ge1$, as desired.
\end{proof}
Let $a_{\left(i,j\right),\left(i',j'\right)}$ be the entry of Hasse-Witt
matrix corresponding to coefficient of $\omega_{i',j'}$ when $\mathscr{C}\left(\omega_{i,j}\right)$
is written as a linear combination of the basis elements. In particular,
$a_{\left(i,j\right),\left(i,j\right)}$ would refer to a diagonal
entry on the Hasse-Witt matrix. We adopt the following notation for
convenience. 
\begin{defn}
\begin{singlespace}
Define 
\[
\begin{bmatrix}a\\
b
\end{bmatrix}:=\begin{cases}
\binom{a}{b}, & \textrm{if }a,b\in\mathbb{Z}_{\ge0};\\
0, & \textrm{otherwise}.
\end{cases}
\]
\end{singlespace}

and

\[
\begin{bmatrix}a\\
b_{1},b_{2},\ldots,b_{k}
\end{bmatrix}:=\begin{cases}
\binom{a}{b_{1},b_{2},\ldots,b_{k}}, & \textrm{if }a,b_{i}\in\mathbb{Z}_{\ge0}\textrm{ and }\sum_{i}b_{i}=a;\\
0, & \textrm{otherwise}.
\end{cases}
\]
\end{defn}
\begin{thm}
\begin{singlespace}
\label{thm:The-Hasse-Witt-matrix}The Hasse-Witt matrix of a superelliptic
curve is given by $\left[a_{\left(i,j\right),\left(i',j'\right)}\right]$,
where 
\begin{equation}
a_{\left(i,j\right),\left(i',j'\right)}=\;{\displaystyle \sum_{\mathclap{\begin{array}{c}
{\scriptstyle k_{0}+k_{1}+\ldots+k_{c}=v_{j}=(j-a+p(a-j'))/a,}\\
{\scriptstyle pi'=bv_{j}+ck_{c}+(c-1)k_{c-1}+\ldots+k_{1}+i}
\end{array}}}}\qquad\qquad\begin{bmatrix}(j-a+p(a-j'))/a\\
k_{0},\ldots,k_{c}
\end{bmatrix}m_{0}^{k_{0}}m_{1}^{k_{1}}\ldots m_{c}^{k_{c}}.\label{eq:hasse-witt}
\end{equation}
\end{singlespace}
\end{thm}
\begin{proof}
\begin{singlespace}
Using Lemma \ref{lem:1} and $y^{a}=x^{b}f(x)$, we find that 
\begin{alignat*}{1}
\omega_{i,j} & =\dfrac{x^{i-1}y^{j-1+(a-1)(p-1)}}{y^{p(a-1)}}\\
 & =\dfrac{x^{i-1}y^{p(u_{j}-1)+av_{j}}}{y^{p(a-1)}}\\
 & =\left(\dfrac{y^{u_{j}-1}}{y^{a-1}}\right)^{p}\sum_{k_{0}+k_{1}+\ldots+k_{c}=v_{j}}\begin{pmatrix}v_{j}\\
k_{0},\ldots,k_{c}
\end{pmatrix}m_{0}^{k_{0}}m_{1}^{k_{1}}\ldots m_{c}^{k_{c}}x^{bv_{j}+ck_{c}+(c-1)k_{c-1}+\ldots+k_{1}+i-1}.
\end{alignat*}
Then, using Lemma \ref{lem:cartier lem}, 
\begin{alignat*}{1}
\mathscr{C}(\omega_{i,j}) & =\sum_{\mathclap{\begin{array}{c}
{\scriptstyle k_{0}+k_{1}+\ldots+k_{c}=v_{j},}\\
{\scriptstyle p|bv_{j}+ck_{c}+(c-1)k_{c-1}+\ldots+k_{1}+i}
\end{array}}}\qquad\begin{pmatrix}v_{j}\\
k_{0},\ldots,k_{c}
\end{pmatrix}m_{0}^{k_{0}}m_{1}^{k_{1}}\ldots m_{c}^{k_{c}}\omega_{s_{i,j}(k_{0},k_{1},\ldots,k_{c}),u_{j}}.
\end{alignat*}
Now, Lemmas \ref{lem:2}, \ref{lem:3}, and \ref{lem:4} imply that
$(s_{i,j}(k_{0},\ldots,k_{c}),u_{j})\in\mathcal{N}$, so $\omega_{s_{i,j}(k_{0},\ldots,k_{c}),u_{j}}$
is a basis element. The result follows.
\end{singlespace}
\end{proof}
Using (\ref{eq:C}), we deduce the following.
\begin{thm}
\begin{singlespace}
\label{thm:The-number-of}The number of points on the smooth projective
model of a superelliptic curve defined by $y^{a}=x^{b}\left(m_{c}x^{c}+m_{c-1}x^{c-1}+\ldots+m_{0}\right)$
over $\mathbb{F}_{p}$ is 
\[
\#C(\mathbb{F}_{p})\equiv1-{\displaystyle \sum_{(i,j)\in N}}\left(\;\;\qquad\qquad{\displaystyle \sum_{\mathclap{\begin{array}{c}
{\scriptstyle k_{0}+k_{1}+\ldots+k_{c}=v_{j}=(p-1)(a-j)/a,}\\
{\scriptstyle p|bv_{j}+ck_{c}+(c-1)k_{c-1}+\ldots+k_{1}+i}
\end{array}}}}\qquad\qquad\begin{bmatrix}(p-1)(a-j)/a\\
k_{0},\ldots,k_{c}
\end{bmatrix}m_{0}^{k_{0}}m_{1}^{k_{1}}\ldots m_{c}^{k_{c}}\right)\pmod*{p},
\]
where $p$ is a prime greater than $a$ and $b+c.$
\end{singlespace}
\end{thm}
\begin{cor}
The number of points on the curve $C:y^{a}=m_{c}x^{b+c}+m_{0}x^{b}$
is 
\begin{equation}
\#C(\mathbb{F}_{p})\equiv1-{\displaystyle \sum_{(i,j)\in N}\begin{bmatrix}(p-1)(a-j)/a\\
(p-1)(ai+bj-ab)/(ac)
\end{bmatrix}}m_{0}^{k_{0}}m_{c}^{k_{c}}\pmod*{p},\label{eq:number of points}
\end{equation}
where $k_{0}=\frac{(p-1)\left((a-j)(b+c)-ai\right)}{ac}$ and $k_{c}=\frac{(p-1)(ai+bj-ab)}{ac}.$
\end{cor}
\begin{rem}
\label{A-rather-interesting}A rather interesting number-theoretic
application of Theorem \ref{thm:The-number-of} tells us the number
of solutions to the Diophantine equation $y^{a}=h(x)=r_{a}x^{a}+r_{a-1}x^{a-1}+\ldots+r_{1}x+r_{0}$
modulo $p$. Define $\textrm{ind}_{\zeta_{p}}(m)$ to be the smallest
positive integral power of some primitive root $\zeta_{p}$ of $\mathbb{F}_{p}$
such that $\zeta_{p}^{\textrm{ind}_{\zeta_{p}}(m)}\equiv m\pmod*{p}$.
Then, $y^{a}=h(x)$ has exactly $p+1-w\left(r_{a}\right)$ solutions
$(x,y)\in\mathbb{F}_{p}$ if at least one of $r_{a},r_{a-1}$ is nonzero
in $\mathbb{F}_{p}$, at least one of $r_{1},r_{0}$ is nonzero in
$\mathbb{F}_{p}$, $\gcd(a,p-1)=1$, $p>16g^{2}$, and $h(x)$ is
square-free, where $g$ is the genus of curve $C$ associated with
the Diophantine equation and $w\left(r_{a}\right)$, the number of
$a$th roots of $r_{a},$ is given by 
\[
w\left(r_{a}\right)=\begin{cases}
0, & \textrm{ if \ensuremath{r_{a}} is not a residue of degree \ensuremath{a} modulo }p;\\
\gcd\left(\textrm{ind}_{\zeta_{p}}(r_{a}),p\right), & \textrm{otherwise}.
\end{cases}
\]

It is worth noting that this can be extended to any Diophantine equation
of the form \\
$y^{a}=x^{b}\left(m_{c}x^{c}+m_{c-1}x^{c-1}+\ldots+m_{0}\right).$
However, the number of solutions to this equation is not the same
as $\#C(\mathbb{F}_{p})$: Theorem \ref{thm:The-number-of} gives
the number of points on the smooth projective model of the superelliptic
curve defined by this equation, which is usually different from the
number of points in the affine model. Thus, for the general case,
a tiny correction related to the singular points has to be made to
account for the discrepancy between $\#C(\mathbb{F}_{p})$ and the
number of solutions to the equation $y^{a}=x^{b}\left(m_{c}x^{c}+m_{c-1}x^{c-1}+\ldots+m_{0}\right)$. 
\end{rem}

\section{Fast Point-Counting Algorithms}

Given our work in Section 3, we have reduced the problem of computing
the Hasse-Witt (and hence $\#C\left(\mathbb{F}_{p}\right)$ and $\#J_{C}\left(\mathbb{F}_{p}\right)\bmod{p}$)
to computing certain multinomial coefficients modulo $p$.

\subsection{Computing $\#C\left(\mathbb{F}_{p}\right)$ for Specific Trinomial
Superelliptic Curves\label{subsec:specific trinomial}}

We consider a specific class of superelliptic curves, namely those
of the form $y^{a}=m_{c}x^{b+c}+m_{0}x^{b}$. To compute the number
of points on these specific trinomial superelliptic curves, we generalize
Theorem 3.5 (which applies to curves of the form $y^{2}=x^{8}+c$
and $y^{2}=x^{7}-cx$) in Fité and Sutherland's paper \cite{sutherland}
to apply to this much more general class of curves.
\begin{thm}
\label{thm:number of points} Given a curve $C:y^{a}=m_{c}x^{b+c}+m_{0}x^{b}$
and a prime $p>16g^{2},$ where $g$ is the genus of $C$ and $\gcd(ac,p-1)=e=3,4,6,\textrm{ or }8,$
there exists an algorithm that computes $\#C\left(\mathbb{F}_{p}\right)$ 
\end{thm}
\begin{itemize}
\item \textit{probabilistically in }$O(\textrm{M}(\log p)\log p)$ \textit{time;}
\item \textit{deterministically in }$O\left(\textrm{M}(\log p)\log^{2}p\log\log p\right)$
\textit{time, assuming GRH;}
\item \textit{deterministically in }$O\left(\textrm{M}(\log^{3}p)\frac{\log^{2}p}{\log\log p}\right)$
\textit{time.}
\end{itemize}
\textit{In addition, given a positive integer $N$ we can compute
$\#C\left(\mathbb{F}_{p}\right)$ for all primes $p$, such that $3<p\le N$,
deterministically in }$O\left(N\textrm{M}\left(\log N\right)\right)$
\textit{time.}
\begin{proof}
We give such an algorithm. It consists of three steps: 1) computing
{\scriptsize{}${\scriptstyle \begin{bmatrix}(p-1)(a-j)/a\\
(p-1)(ai+bj-ab)/(ac)
\end{bmatrix}}$} modulo $p$ for all $(i,j)\in\mathcal{N}$, 2) finding $\#C\left(\mathbb{F}_{p}\right)\bmod{p}$
using (\ref{eq:number of points}), and 3) determining $\#C\left(\mathbb{F}_{p}\right)$
using (\ref{eq:hasse bound}). Assuming $a$ and $c$ are fixed, the
second step runs in $O(\textrm{M}(\log p)\log p)$ time and the third
step runs in at most $O(\textrm{M}(\log p))$ time. In what follows,
we will describe and determine the time complexity of the first step.
We shall then see that, in similar fashion to the proof of Theorem
3.5 in Fité and Sutherland's paper \cite{sutherland}, the desired
runtimes follow from the time complexities of the first step.

Let $\gcd(ac,p-1)=e$ and $ac=ed,p-1=ef$, where $d,e,f\in\mathbb{N}$.
Note that $\gcd(d,f)=1$. Then, we have
\[
\begin{bmatrix}(p-1)(a-j)/a\\
(p-1)(ai+bj-ab)/(ac)
\end{bmatrix}=\begin{bmatrix}(p-1)(ac-jc)/(ac)\\
(p-1)(ai+bj-ab)/(ac)
\end{bmatrix}=\begin{bmatrix}(ed-jc)f/d\\
(ai+bj-ab)f/d
\end{bmatrix}.
\]
Using congruences for binomial coefficients of the form $\binom{rf}{sf}$
modulo primes of the form $p=ef+1$, we can compute $\#C\left(\mathbb{F}_{p}\right)$
extremely quickly. In particular, when $e=3,4,6,\textrm{ or }8$,
it is possible to compute $\binom{rf}{sf}$ modulo $p$ efficiently
for $1\le s<r\le e-1$. Note that, by definition,
\[
\begin{bmatrix}(ed-jc)f/d\\
(ai+bj-ab)f/d
\end{bmatrix}=\begin{cases}
\binom{(ed-jc)f/d}{(ai+bj-ab)f/d}, & \textrm{if }d|ed-jc\textrm{ and }d|ai+bj-ab;\\
0, & \textrm{otherwise}.
\end{cases}
\]
We can then compute this expression quickly using results from Hudson
and Williams' work \cite{hudson}. 

We now give explicit congruence relations for the aforementioned binomial
coefficients $\binom{rf}{sf}.$ We will also use similar notation
to that of Hudson and William \cite{hudson}: $p=a_{4}^{2}+b_{4}^{2}$
for $e=4$, $p=a_{3}^{2}+3b_{3}^{2}$ for $e=3,6$, $p=a_{8}^{2}+2b_{8}^{2}$
for $e=8$, $a_{4}\equiv1\pmod*{4}$, $a_{3}\equiv1\pmod*{3}$, and
$a_{8}\equiv1\pmod*{4}$. For each $e$, define ``representative
binomial coefficients'' as the binomial coefficients sufficient to
generate all of the other binomial coefficients through trivial congruences
\cite{hudson}. Table \ref{tab:The-entry-that-1} summarizes the congruence
relations for the representative binomial coefficients for $e=3,4,6,$
and 8 (let $\overline{b_{3}}\in\mathbb{Z}$ be $b_{3}\bmod{3}$).
\begin{table}[h]
\begin{centering}
\begin{tabular}{|c|c|c|c|c|}
\hline 
$\binom{af}{bf}$ & $e=3$ & $4$ & $6$ & $8$\tabularnewline
\hline 
\hline 
$\binom{2f}{f}$ & $\begin{cases}
2a_{3} & \textrm{if }\overline{b_{3}}=0\\
-a_{3}-3b_{3} & \textrm{if }\overline{b_{3}}=1\\
-a_{3}+3b_{3} & \textrm{if }\overline{b_{3}}=2
\end{cases}$ & $2a_{4}$ & $\begin{cases}
2\left(-1\right)^{f}a_{3} & \textrm{if }\overline{b_{3}}=0\\
\left(-1\right)^{f}\left(-a_{3}+3b_{3}\right)b_{3} & \textrm{if }\overline{b_{3}}=1\\
\left(-1\right)^{f}\left(-a_{3}-3b_{3}\right)b_{3} & \textrm{if }\overline{b_{3}}=2
\end{cases}$ & $\left(-1\right)^{b_{4}/4}2a_{8}$\tabularnewline
\hline 
$\binom{3f}{f}$ & - & - & $2a_{3}$ & $\left(-1\right)^{f+b_{4}/4}2a_{4}$\tabularnewline
\hline 
$\binom{4f}{f}$ & - & - & - & $\left(-1\right)^{f}2a_{8}$\tabularnewline
\hline 
$\binom{5f}{2f}$ & - & - & - & $\left(-1\right)^{f+b_{4}/4}2a_{8}$\tabularnewline
\hline 
\end{tabular}
\par\end{centering}
\caption{The entry that corresponds to the row with representative binomial
coefficient $\binom{af}{bf}$ and column $e$ is congruent to the
$\binom{af}{bf}\pmod*{ef+1}$ . For example, we have $\binom{3f}{f}\equiv2a_{3}\pmod*{6f+1}$.\label{tab:The-entry-that-1}}
\end{table}

We have effectively reduced the problem of computing binomial coefficients
modulo $p$ to a problem of solving Diophantine equations of the form
$x^{2}+dy^{2}=m$. We can then employ Cornacchia's algorithm \cite{hudson}
to find $a_{3},a_{4},a_{8},b_{3},b_{4},$ and $b_{8}$. However, before
doing so, we must find a square root of $-d$ modulo $m$. 

There are three well-known methods of efficiently computing a square
root modulo $p$: the Cipolla-Lehmer, Tonelli-Shanks, and Schoof algorithms
\cite{sutherland}. Define $\textrm{M}(n)$ to be the time needed
to compute the product of two $n$-digit numbers. For practical applications,
we can take $\textrm{M}(n)$ to be $O\left(n\log n\log\log n\right)$
or $O\left(n\left(\log n\right)2^{\log^{*}n}\right)$ \cite{fast mult}.
The probabilistic Cipolla-Lehmer approach involves factorization of
a quadratic over $\mathbb{F}_{p}[x]$ and takes $O(\textrm{M}(\log p)\log p)$
time. The deterministic Tonelli-Shanks approach requires computing
a generator for the maximal 2-subgroup of $\mathbb{F}_{p}^{*}$ and
takes $O\left(\textrm{M}(\log p)\log^{2}p\log\log p\right)$ time,
assuming the generalized Riemann hypothesis (GRH). The Schoof approach
takes advantage of properties of certain elliptic curves and requires
$O\left(\textrm{M}(\log^{3}p)\frac{\log^{2}p}{\log\log p}\right)$
time. Fité and Sutherland discuss these algorithms in much greater
detail in their study of computing $\#C\left(\mathbb{F}_{p}\right)$
for certain hyperelliptic curves \cite{sutherland}. 

We can now use Cornacchia's algorithm: given a Diophantine equation
of the form $x^{2}+dy^{2}=m$ and a square root of $-d$ modulo $m$,
such that $\gcd(d,m)=1$, we can compute a solution $(x,y)$ (if one
exists). It turns out that the time complexity of Cornacchia's algorithm
is essentially the same as that of the Euclidean algorithm \cite{sutherland}.

To compute $a_{3},a_{4},a_{8},b_{3},b_{4},$ and $b_{8}$, we can
first use Cornacchia's algorithm to find a solution $(x,y)$ to $x^{2}+dy^{2}=p$,
where $d$ is $1,2,$ or 3. Taking the equation $x^{2}+y^{2}=p$ modulo
4, it is clear that at least one of $x$ and $y$ is odd. Without
loss of generality, say $x$ is odd. We can replace $x$ with $-x$,
so we can assume $x\equiv1\pmod{4}$ and let $a_{4}=x$. Likewise,
if we take $x^{2}+3y^{2}=p$ modulo 3, we can set $a_{3}\equiv1\pmod{3}$,
since $x$ is either 1 or $-1$ modulo 3 and if we take $x^{2}+2y^{2}=p$
modulo 2, $x$ must be odd, so we can set $a_{8}\equiv1\pmod{4}$
by letting $a_{8}$ equal $x$ or $-x$ modulo 4. 
\end{proof}
\begin{rem}
It is worth noting that Pila's generalization of Schoof's algorithm
can compute \textit{$\#C\left(\mathbb{F}_{p}\right)$} and the zeta
function in time polynomial in $\log p$; however, the time complexity
is a large power of $\log p$ that grows quickly with the genus. Even
in genus 2, the most efficient algorithms using Pila's approach take
$O\left(\log^{6+o(1)}p\right)$ time \cite{Pila}, but the algorithm
we have just described performs much better.
\end{rem}
\begin{example}
\label{Say--(99-bit)} Say $p=564819669946735512444543556507$ (99-bit)
and $C:y^{4}=x^{11}+x^{8}$ (note that $\gcd(564819669946735512444543556506,12)=6$.
The algorithm given in Theorem \ref{thm:number of points} computed
\[
\#C\left(\mathbb{F}_{p}\right)=564819669946736263601275822996
\]
in 66.2 ms using SageMath Version 7.3 on an Intel Celeron 2955U processor
at 1.4 GHz. In comparison, a brute-force algorithm to compute the
number of points for the same curve with $p=10133$ (14-bit) instead
took over six hours to run. 
\end{example}
It should be noted that this algorithm can be extended to find $\#J_{C}\left(\mathbb{F}_{p}\right)$
for genus two and three curves of the same form, with the additional
condition that $a|p-1$. It turns out that Hasse-Witt matrix $A_{p}\left(C\right)$
of such a curve is diagonal; it is not hard to see that $s_{i,j}\left(k_{0},\ldots,k_{c}\right)$
can only be $i$ and $u_{j}$ can only be $j$. If $A_{p}\left(C\right)$
is diagonal, $A_{p}\left(C\right)-I$ is also diagonal, so (\ref{eq:J})
tells us that 
\begin{equation}
\#J_{C}\left(\mathbb{F}_{p}\right)\equiv(-1)^{g}\prod_{(i,j)\in\mathcal{N}}\left(\begin{bmatrix}(p-1)(a-j)/a\\
(p-1)(ai+bj-ab)/(ac)
\end{bmatrix}m_{0}^{k_{0}}m_{c}^{k_{c}}-1\right)\pmod*{p},\label{eq:J mod p}
\end{equation}
where $k_{0}=\frac{(p-1)\left((a-j)(b+c)-ai\right)}{ac}$ and $k_{c}=\frac{(p-1)(ai+bj-ab)}{ac}$.

For genus two curves, the Hasse-Weil bounds and $\#J_{C}\left(\mathbb{F}_{p}\right)$
modulo $p$ gives us a maximum of five candidates for $\#J_{C}\left(\mathbb{F}_{p}\right)$,
so we can compute $\#J_{C}\left(\mathbb{F}_{p}\right)$ in the same
time as that of $\#C\left(\mathbb{F}_{p}\right)$, because arithmetic
in the Jacobian has a comparatively lower time complexity \cite{japanese,poonen2}.
For genus three curves, the Hasse-Weil bounds and $\#J_{C}\left(\mathbb{F}_{p}\right)$
modulo $p$ gives us a maximum of $\left\lceil 40\sqrt{p}\right\rceil $
candidates for $\#J_{C}\left(\mathbb{F}_{p}\right)$, so we can compute
$\#J_{C}\left(\mathbb{F}_{p}\right)$ in $O\left(\sqrt{p}\right)$
time \cite{poonen2,picard genus 3}. In comparison, Harvey has written
down an improvement of Kedlaya's point-counting algorithm to compute
$\#J_{C}\left(\mathbb{F}_{p}\right)$ for hyperelliptic curves in
$O\left(p^{1/2+o(1)}\right)$ time \cite{Harvey Kedlaya}. Our algorithm,
however, is not efficient for curves of higher genus.

\subsection{Computing $\#C\left(\mathbb{F}_{p}\right)$ and $A_{p}\left(C\right)$
for General Superelliptic Curves\label{subsec:Simplifying-the-Hasse-Witt}}

We describe an algorithm to simplify the entries of the Hasse-Witt
matrix modulo $p$ for a general superelliptic curve and compute $\#C\left(\mathbb{F}_{p}\right)$
and $A_{p}\left(C\right)$ efficiently. More generally, Harvey has
shown how to compute the zeta function of any algebraic variety over
$\mathbb{F}_{p}$ (including superelliptic curves) in $O(\sqrt{p}\log^{2+\epsilon}p)$
time \cite{Harvey general curve}, using Bostan et al.'s algorithm
\cite{Boston}. To do so, Harvey describes a new trace formula and
a deformation recurrence; on the other hand, we find the zeta function
modulo $p$ by reducing the Hasse-Witt matrix modulo $p$.

Throughout this section, we will assume that $b=0$, i.e. the equation
of superelliptic curve is $y^{a}=f(x)=m_{c}x^{c}+\ldots+m_{0}$, where
$m_{0}\ne0$ (note that we can perform a simple coordinate change
to (\ref{eq:superelliptic})). We summarize the algorithm as follows.
\begin{thm}
\label{thm:For-a-general} Given a curve $C:y^{a}=m_{c}x^{c}+\ldots+m_{0}$
and a prime $p>16g^{2},$ where $g$ is the genus of $C,$ there exists
an algorithm that computes $\#C\left(\mathbb{F}_{p}\right)$ in \textup{$O\left(\textrm{M}\left(\sqrt{p}\log p\right)\right)$}
time, assuming fixed $a$ and $c$. 
\end{thm}
\begin{proof}
We give such an algorithm. It consists of three steps: 1) adapting
Bostan et al.'s algorithm \cite{Boston} to simplify the entries of
the Hasse-Witt matrix, 2) adding up the diagonal entries of the Hasse-Witt
matrix to find $\#C\left(\mathbb{F}_{p}\right)\bmod{p}$, and 3) determining
$\#C\left(\mathbb{F}_{p}\right)$ using (\ref{eq:hasse bound}). The
second step runs in $O(\textrm{M}(\log p)\log p)$ time and the third
step runs in at most $O(\textrm{M}(\log p))$ time. In what follows,
we will describe the first step and determine its time complexity.

In Section \ref{sec:Computing-Hasse-Witt-Matrices}, we found an explicit
formula for the Hasse-Witt matrix of any superelliptic curve. Specifically,
(\ref{eq:hasse-witt}) tells us that the problem of simplifying the
Hasse-Witt matrix is essentially a question of simplifying certain
multinomial coefficients modulo $p$. We can reinterpret this as a
problem of computing coefficients of certain powers of $x$ in $f^{v_{j}},$
where $v_{j}$ is defined in Lemma \ref{lem:2}. For convenience,
let $f_{i,j}$ denote the coefficient of $x^{i}$ in $f^{j}$.

Setting $b=0$ in (\ref{eq:hasse-witt}) and using the multinomial
theorem (note that $pi'-i=bv_{i}+ck_{c}+(c-1)k_{c-1}+\ldots+k_{1}$)
imply that 
\[
a_{\left(i,j\right),\left(i',j'\right)}=f_{pi'-i,v_{j}},
\]
where $a_{\left(i,j\right),\left(i',j'\right)}$ is the coefficient
of $\omega_{\left(i',j'\right)}$ when $\omega_{\left(i,j\right)}$
is written as a linear combination of the basis of regular differentials.

We can now use Bostan et al.'s algorithm \cite{Boston} for quickly
computing coefficients of large powers of a polynomial. In their paper,
they apply their algorithm to point-counting for hyperelliptic curves.
We generalize their method to superelliptic curves; the extension
is quite similar to the original method for hyperelliptic curves,
so we will omit the finer details \cite{picard curve}. We can say
that 
\[
f\left(f^{v_{j}}\right)'-v_{j}f'f^{v_{j}}=0.
\]
Comparing coefficients of $x^{k}$ on both sides, we find that 
\[
\left(k+1\right)m_{0}f_{k+1,v_{j}}+\left(k-v_{j}\right)m_{1}f_{k,v_{j}}+\ldots+\left(\left(k-c+1\right)-cv_{j}\right)f_{k-c+1,v_{j}}=0.
\]
Let 
\[
A(k)=\begin{pmatrix}0 & 1 & 0 & \ldots & 0\\
0 & 0 & 1 & \ldots & 0\\
\vdots & \vdots & \vdots & \ddots & \vdots\\
0 & 0 & 0 & \ddots & 1\\
\frac{m_{c}\left(cv_{j}-(k-c)\right)}{m_{0}k} & \frac{m_{c-1}\left((c-1)v_{j}-(k-c+1)\right)}{m_{0}k} & \ldots & \ldots & \frac{m_{1}\left(v_{j}-(k-1)\right)}{m_{0}k}
\end{pmatrix}.
\]

Now, if we define the vector $U_{k}=\left[f_{k-c+1,v_{j}},\ldots,f_{k-1,v_{j}},f_{k,v_{j}}\right]^{t}$
and initial vector $U_{0}=\left[0,\ldots,0,m_{0}^{v_{j}}\right]^{t},$
we have $U_{k}=A(k)U_{k-1}$ for $k\in\mathbb{N}$. Thus, $U_{k}=A(k)A(k-1)\ldots A(1)U_{0}$.
Thus, to compute $a_{\left(i,j\right),\left(i',j'\right)}=f_{pi'-i,v_{j}}$,
it suffices to find $U_{pi'-1}$ (note that $f_{pi'-i}$ is the $i$th
entry of $U_{pi'-1}$ from the right). Bostan et al.'s algorithm \cite{Boston}
cannot be applied directly because the entries of $A(k)$ are rational
functions, so we must modify the $U_{i}$ slightly: let $V_{k}=m_{0}^{k}k!U_{k}$
and $B(k)=m_{0}kA(k)$. Then, $V_{k}=B(k)B(k-1)\ldots B(1)V_{0}$
and the entries of the matrix $B(k)$ are polynomials in $\mathbb{F}_{p}[k].$ 

Using either of the algorithms from Theorems 14 and 15 from Bostan
et al.'s paper \cite{Boston}, we can compute $U_{pi'-1}$ for all
$i'$ and a particular $j$ in $O\left(\textrm{M}\left(\sqrt{p}\log p\right)\right)$
time. The desired runtime follows from applying this repeatedly to
all $j$ (there are at most $a$ possible values of $j$ because $1\le j\le a$).
\end{proof}
\begin{rem}
Given that $p>g$, the proof of Theorem \ref{thm:For-a-general} also
shows that we can simplify $A_{p}\left(C\right)$, the Hasse-Witt
matrix modulo $p$, in $O\left(\textrm{M}\left(\sqrt{p}\log p\right)\right)$
time.
\end{rem}
It should be noted that this algorithm can be extended to find $\#J_{C}\left(\mathbb{F}_{p}\right)$
for genus two and three curves of the same form. We can compute $\#J_{C}\left(\mathbb{F}_{p}\right)$
in $O\left(\textrm{M}\left(\sqrt{p}\log p\right)\right)$ time for
genus two curves and in $O\left(\textrm{M}\left(\sqrt{p}\log p\right)\right)$
time for genus three curves (refer to the discussion at the end of
Subsection \ref{subsec:specific trinomial}).

\section{Conclusion}

In this paper, we employed the Cartier operator to derive the entries
of the Hasse-Witt matrix of a superelliptic curve in terms of multinomial
coefficients, as shown in Theorem \ref{thm:The-Hasse-Witt-matrix}.
We then described and determined the time complexities of efficient
point-counting algorithms for both specific trinomial and general
curves. In particular, for specific trinomial superelliptic curves,
we reduced the problem of point-counting to a problem of solving Diophantine
equations of the form $x^{2}+dy^{2}=m$ and computed $\#C\left(\mathbb{F}_{p}\right)$
probabilistically in $O(\textrm{M}(\log p)\log p)$ time, deterministically
in $O\left(\textrm{M}(\log p)\log^{2}p\log\log p\right)$ (assuming
the generalized Riemann hypothesis), and deterministically in $O\left(\textrm{M}(\log^{3}p)\frac{\log^{2}p}{\log\log p}\right)$
time, assuming $p>16g^{2}$ and $\gcd(ac,p-1)=3,4,6\textrm{ or }8.$
We also described how to compute $\#J_{C}\left(\mathbb{F}_{p}\right)$
with the same running times for $g=2$ and with running time $O\left(\sqrt{p}\right)$
if $g=3$, assuming that $a$ also divides $p-1$. As shown in Example
\ref{Say--(99-bit)}, we implemented and demonstrated the efficiency
of the algorithm given in Theorem \ref{thm:number of points} to compute
$\#C\left(\mathbb{F}_{p}\right)$ for the superelliptic curve with
affine model $y^{4}=x^{11}+x^{8}$ over $\mathbb{F}_{564819669946735512444543556507}$.
For general superelliptic curves, we employed Bostan et al.'s algorithm
\cite{Boston} to compute $\#C\left(\mathbb{F}_{p}\right),$ $A_{p}\left(C\right),$
and $\#J_{C}\left(\mathbb{F}_{p}\right)$ in $O\left(\textrm{M}\left(\sqrt{p}\log p\right)\right)$
time, assuming $p>16g^{2}$ (and additionally $g=2\textrm{ or 3}$
for $\#J_{C}\left(\mathbb{F}_{p}\right)$). To the best of our knowledge,
we have written down the fastest point-counting algorithms for specific
trinomial superelliptic curves.

In the future, we hope to write down explicit algorithms for computing
$\#J_{C}\left(\mathbb{F}_{p}\right)$, which will plausibly involve
finding the coefficients of the characteristic polynomial of Frobenius.
We also plan on implementing the algorithm given in Theorem \ref{thm:For-a-general}
through a simple modification of the original algorithm described
by Bostan et al. for hyperelliptic curves \cite{Boston}. We also
hope to improve the efficiencies of our algorithms. Harvey et al.
discussed how the characteristic polynomials of Frobenius for genus
three hyperelliptic curves can be computed in $p^{1/4+o(1)}$ time
by ``lifting'' the characteristic polynomials modulo $p$ \cite{genus three fast arithemtic};
it may be possible to extend these methods to the curves of interest
in this paper. For studying Sato-Tate distributions, we will also
look into extending Harvey et al.'s approach for efficiently computing
$\#C\left(\mathbb{F}_{p}\right)$ and Hasse-Witt matrices for many
$p$ at once \cite{sutherland hasse-witt}. Finally, we hope to extend
our methods to non-superelliptic curves. Although not discussed in
this paper, we have also looked into interpreting the Hasse-Witt matrix
through the lens of Cech cohomology. We believe that such methods
are easier to generalize, especially for trinomial curves of the form
$y^{m_{1}}=c_{1}x^{n_{1}}+c_{2}x^{n_{2}}y^{m_{2}}.$

\section{Acknowledgments}

We would like to thank Professor Andrew Sutherland from MIT for suggesting
this project and for providing valuable advice. Moreover, none of
this research would have been possible without the help of Dr. Tanya
Khovanova, Professor Pavel Etingof, and Dr. Slava Gerovitch from MIT
for organizing PRIMES-USA, a math research program for high school
juniors.

\end{document}